\documentclass{amsart}
\usepackage[pdfusetitle,colorlinks,citecolor=blue,urlcolor=black,linkcolor=black]{hyperref}
\usepackage{amsfonts,amssymb,amsmath,amsthm,stmaryrd,cmap,verbatim}
\usepackage{enumitem}
\usepackage{nicefrac}
\newtheorem{maintheorem}{Main Theorem}[]
\newtheorem{theorem}{Theorem}[section]
\newtheorem{claim}{Claim}[theorem]
\newtheorem{lemma}[theorem]{Lemma}

\newtheorem{corollary}[theorem]{Corollary}

\theoremstyle{definition} 
\newtheorem{definition}[theorem]{Definition}

\theoremstyle{remark}
\newtheorem{remark}[theorem]{Remark}

\DeclareMathOperator{\cf}{cf}
\DeclareMathOperator{\dom}{dom}
\DeclareMathOperator{\Succ}{Succ}
\DeclareMathOperator{\otp}{otp}
\DeclareMathOperator{\col}{Col}
\DeclareMathOperator{\Add}{Add}
\DeclareMathOperator{\rfl}{Refl}
\DeclareMathOperator{\ap}{AP}

\newcommand\s{\subseteq}
\newcommand\br{\blacktriangleright}

\renewcommand\mid{\mathrel{|}\allowbreak}
\def\Mid{\mathrel{\bigm|}}
\newcommand\dprime{{\prime\prime}}
\newcommand\trle{\trianglelefteq}
\newcommand\ra{\rightarrow}

\newcommand*\axiomfont[1]{\textsf{\textup{#1}}}

\newcommand\gch{\axiomfont{GCH}}
\newcommand\sch{\axiomfont{SCH}}

\author{Inbar Oren}
\address{Einstein Institute of Mathematics, The Hebrew University of Jerusalem, Israel}
\email{inbar.oren2@mail.huji.ac.il}

\title{Stationary reflection using ancestrally forced conditions}
\date{Preprint as of \today.}

\begin{document}
\begin{abstract} 
We introduce a new method for obtaining models of stationary reflection at the successor of a singular cardinal of any cofinality. In particular, starting from a cardinal $\kappa$ which is $\kappa^+$-supercompact, we get a model of countable simultaneous reflection at $\aleph_{\omega_1+1}$, improving the result of \cite{BenhamouSinapova}. 
\end{abstract}
\maketitle

\section*{Introduction}
The consistency of compactness principles at accessible cardinals is a well-studied area. In the works of Baumgartner \cite{Baumgartner76} and later Magidor and Shelah \cite{Magidor1982, Shelah:204, Shelah:351}, the specific case of reflection of stationarity of subsets of regular cardinals was identified as a central barrier towards other compactness principles, in particular at the successors of singular cardinals. 

While the consistency strength of stationary reflection at successor of regular or inaccessible cardinals is fairly well understood \cite{Jensen72, Shelah:99, Shelah367}, there is still a large gap between the lower and upper bounds for the consistency strength of the parallel phenomenon at successors of singular cardinals. Several recent works have studied stationary reflection at the successors of singular cardinals, dealing both with the consistency strength and the comptibility of stationary reflection with other combinatorial principles.

A notable development in the questions of compatibility of stationary reflection with the failure of $\sch$ is the framework of $\Sigma$-Prikry forcing  \cite{SigmaPrikry1,SigmaPrikry2,SigmaPrikry3}--- in which stationary reflection at a successor of a singular cardinal is obtained via a sophisticated iteration of Prikry-type forcing notions. These results typically require very strong large cardinal assumptions, such as limits of supercompact cardinals to manage the preservation of stationary sets across the forcing.
For singular cardinals of uncountable cofinality, Benhamou and Sinapova have successfully applied this machinery to force reflection principles at $\kappa^+$, e.g. $\aleph_{\omega_1+1}$, with the failure of $\sch$ \cite{BenhamouSinapova}. 

In a parallel development concerning reductions in the upper bound on the consistency strength, Hayut and Unger showed that Magidor's reflection can be obtained from significantly weaker hypotheses than previously believed \cite{HAYUTUnger2020}. This method was extended later in \cite{BenNeriaHayutUnger} to incorporate the failure of $\sch$.

We aim to generalize their result, so that the same large cardinal hypothesis would allow us to obtain stationary reflection at the successor of a singular with \emph{uncountable} cofinality. 
It seems that Hayut and Unger's machinery of iterated ultrapowers is problematic at uncountable cofinalities. Thus, achieving the main result of this paper required us to strip this machinery down and create a Prikry-type forcing notion that achieves the same result from the same large cardinal assumptions.

Our method, ancestral forcing, permits us to use almost the same assumptions as Hayut and Unger to get stationary reflection at a successor of a singular cardinal with uncountable cofinality. 
Moreover, our method can achieve much more and in particular we can obtain the consistency of countable simultaneous reflection at $\aleph_{\omega_1+1}$ assuming the existence of a cardinal $\kappa$ that is $\kappa^+$-supercompact.

\subsection*{Notations and definitions.}
Our notations are mostly standard. 

Let $\kappa<\lambda$. We denote $E^\lambda_\kappa=\{\alpha<\lambda\mid \cf(\alpha)=\kappa\}$, and $E^\lambda_{<\kappa}=\{\alpha<\lambda\mid \cf(\alpha)<\kappa\}$.

\begin{definition}
    Let $\lambda$ be a limit ordinal and let $S\s \lambda$ be stationary. 
    \begin{enumerate}
        \item We say that  $\rfl(S)$ holds iff for every $T\s S$ stationary there is $\delta<\lambda$ such that $T\cap \delta$ is stationary in $\delta$.
        \item Let $\kappa\le \lambda$. We say that $\rfl_{<\kappa}(S)$ holds iff for every $\mathcal S\s \mathcal P(S)$ such that $|\mathcal S|<\kappa$ and every $A\in \mathcal S$ is a stationary subset of $S$, there is some $\delta<\lambda$ such that for every $A\in\mathcal S$, $A\cap\delta$ is stationary in $\delta$.
        \item Let $\mathcal S\s \mathcal P(\lambda)$ be such that each $S\in\mathcal S$ is stationary. We denote $T(\mathcal S):=\{\delta<\lambda\mid \forall S\in \mathcal S \text{ reflects at }\delta\}$.
    \end{enumerate} 
\end{definition}

Let $\mu$ be a regular cardinal. Let $I[\mu]$ denote Shelah's Approachability Ideal, as defined in \cite{Sh:108}; see also \cite{Eisworthhandbook}.

We say that $\ap_\mu$ holds iff $\mu\in I[\mu]$.

We assume familiarity with forcing techniques, and in particular with Prikry-type forcings (e.g., Magidor forcing and the related notions of Mitchell order). We refer the reader to \cite{Jech2003,GitikHandbook}.

\subsection*{Paper overview.}

Let us describe the structure of this paper:
\begin{enumerate}
\item In Section \ref{sec: uncountable construction} we construct our main forcing notions $\mathbb M^\ast$ and $\mathbb R^\ast$, where both of which are essential to the proofs of our main theorems.

In Subsection \ref{subsec: M*} we define the main forcing $\mathbb M^\ast$; in \ref{subsubsec: proj} we observe the essential projection and in ~\ref{sec:prikryproperty} we prove the strong Prikry Property for this forcing and show that it is cardinal-preserving. 
In Subsection \ref{subsec: R*} we define the main forcing $\mathbb R^\ast$ which depends on $\mathbb M^\ast$ and in ~\ref{sec:prikrypropertydown} we prove the strong Prikry Property for this forcing and show it preserves the required cardinals. 
\item In Section \ref{sec: mainthm proofs} we shall prove the main theorems:
\begin{maintheorem}[\ref{main theorem 1}]\label{1st maintheorem}
    Assume $\gch$.
    Let $\rho$ be some regular cardinal.
    Let $\kappa$ be a measurable cardinal with $o(\kappa)\geq \rho$ such that $\rfl(E^{\kappa^+}_{<\kappa})$ holds and is preserved by $\Add(\kappa^+,1)$.
    
    Then, there is a forcing extension in which all cardinals are preserved, $\rfl(\kappa^+)$ holds and $\kappa$ is singular with cofinality $\rho$.  
    
    Moreover, if $\rfl_{<\kappa}(E^{\kappa^+}_{<\kappa})$ holds and is preserved by $\Add(\kappa^+,1)$ then in the forcing extension $\rfl_{<\rho}(\kappa^+)$ holds as well.
\end{maintheorem}

\begin{maintheorem}[\ref{main theorem 2}] \label{2nd mainthorem}
    Assume $\gch$.
    Let $\rho$ be some regular cardinal.
    If in $V$ $\kappa$ is measurable with $o(\kappa)\geq \rho$ and $\rfl(E^{\kappa^+}_{<\kappa})$ holds and is preserved by $\Add(\kappa^+,1)$ and 
    for any stationary $S\s E^{\kappa^+}_{<\kappa}$, 
    there are unboundedly many regular cardinals $\mu < \kappa$ such that $T(\{S\}) \cap E^{\kappa^{+}}_\mu \neq \emptyset$ and one of the following holds:
    
    \begin{enumerate}        
        \item $\mu$ is not a successor of a singular cardinal, 
        \item $\ap_{\mu}$ holds. 

    \end{enumerate}
    
    Then there is a forcing extension in which $\kappa=\aleph_{\rho}$, $\kappa^+=\aleph_{\rho+1}$ and in which $\rfl(\aleph_{\rho+1})$ holds.  
\end{maintheorem}

\begin{maintheorem}[\ref{main theorem 3}] \label{3rd mainthorem}
    Assume $\gch$.
    Let $\rho$ be some regular cardinal.
    If in $V$ $\kappa$ is measurable with $o(\kappa)\geq \rho$ and $\rfl_{<\kappa}(E^{\kappa^+}_{<\kappa})$ holds and is preserved by $\Add(\kappa^+,1)$ and 
    for any collection $\mathcal S\s\mathcal P(E^{\kappa^+}_{<\kappa})$ consisting $<\kappa$ stationary subsets of $E^{\kappa^+}_{<\kappa}$, 
    there are unboundedly many regular cardinals $\mu < \kappa$ such that $T(\mathcal S) \cap E^{\kappa^{+}}_\mu \neq \emptyset$ one of the following holds:

    \begin{enumerate}        
        \item $\mu$ is not a successor of a singular cardinal, 
        \item $\ap_{\mu}$ holds. 
    \end{enumerate}
    Then there is a forcing extension in which 
    $\kappa=\aleph_{\rho}$, $\kappa^+=\aleph_{\rho+1}$ and in which $\rfl_{<\rho}(\aleph_{\rho+1})$ holds.
\end{maintheorem}
In subsection \ref{subsec: downto construction} we construct a hybrid forcing notion. In \ref{sec:prikrypropertydown} we prove the strong Prikry Property for this forcing, and in subsection \ref{subsec: downto getting stat} we prove the main theorem of this section.
\end{enumerate}

\section{Introducing forcing notions with ancestrally forced conditions.} 
\label{sec: uncountable construction}
In this section we will construct our main forcing notions $\mathbb M^\ast$ and $\mathbb R^\ast$, Prikry-type forcings that will destroy the stationarity of $E^{\kappa^+}_{\kappa}$ and explore the cardinal structure of their generic extensions. 
Both forcing notions are essential for the proofs of our main theorems, where the model required for Main theorem \ref{1st maintheorem} is the generic extension for $\mathbb M^\ast$ and the model required for Main theorems \ref{2nd mainthorem} and \ref{3rd mainthorem} is the generic extension for $\mathbb R^\ast$.

Assume that, in $V$, $\kappa$ is measurable, $o(\kappa)\ge\nu$ for some $\nu<\kappa$ limit ordinal with $\cf(\nu)=\rho$.

Let $\vec{\mathcal{U}} :=\langle U_{i,\alpha}\mid \alpha\le \kappa, i<o^{\mathcal {U}}(\alpha)\rangle$ be a coherent sequence with $o^{\mathcal {U}}(\kappa)=\nu$ and let $\mathbb M :=\mathbb M_{\vec{\mathcal U}}$ be Magidor forcing for this coherent sequence. 
Let $\dot {\mathbb C}$ be a name of a club-shooting forcing in $V^\mathbb M$ that shoots a club disjoint from $(E_\kappa^{\kappa^+})^V$ and let us denote by $\mathbb T(\dot{\mathbb C})$ the term forcing of $\dot {\mathbb C}$. That is,  $\mathbb T(\dot{\mathbb C})= \{\dot c\mid \exists p\in \mathbb M,\, p\Vdash \dot c\in \dot{\mathbb C}\}$.

Let $p^0=\langle(\alpha_0,A_0),\ldots,(\alpha_{n-1},A_{n-1}),A\rangle\in \mathbb M$ and define 
\[x(p^0):=\langle(\alpha_0,A_0),\ldots,(\alpha_{n-1},A_{n-1})\rangle.\] 
Also define $\mathbb Y:=\{x(p^0)\mid p^0\in \mathbb M\}$. 
Notice that $|\mathbb Y|=\kappa$.

For $x\in\mathbb Y$, define
\[\mathbf a(x):= \{ A^x_\beta\cup \{\beta\}\mid \beta\in \dom(x)\}.\]

\subsection{\texorpdfstring{Constructing $\mathbb M^\ast$}{The forcing M*}}\label{subsec: M*} 
Let us define the forcing notion $\mathbb M^\ast$ as the set of conditions of the form
\[ p:=\langle(\alpha_0,A_0),\ldots,(\alpha_{n-1},A_{n-1}),A,c^p\rangle\] 
where:

\begin{enumerate}
    \item $p^0:=\langle(\alpha_0,A_0),\ldots,(\alpha_{n-1},A_{n-1}),A\rangle\in\mathbb M_{\vec {\mathcal U}}$.
    \item $c^p$ is \emph{ancestrally} forced to belong to $\dot{\mathbb C}$. 
    Let $p_\emptyset:= \langle A^p\cup \mathbf a(x(p))\rangle$ then
    \[ p_\emptyset \Vdash \dot c^p\in \dot{\mathbb C}.\]
\end{enumerate}

Let us define relations $\le,\le^\ast$ on $\mathbb M^\ast$ as follows: for $p,q\in\mathbb M^\ast$ we say that 
\begin{itemize}
    \item $q\le p$ ($q$ extends $p$) iff
        \begin{enumerate}
            \item $q^0\le_{\mathbb M_{\vec{\mathcal U}}}p^0$,
            \item $c^p$ is \emph{ancestrally} forced to be an initial segment of $c^q$ i.e.,   
            \[q_\emptyset=\langle A^q\cup \mathbf a(x(q))\rangle \Vdash \dot c^p\trle \dot c^q.\]
        \end{enumerate}
    \item $q\le^\ast p$ ($q$ directly extends $p$) iff $q\le p$ and $q^0\le_{\mathbb M_{\vec{\mathcal U}}}^\ast p^0$.
\end{itemize}

\begin{lemma}
    The order $\le$ on $\mathbb M^\ast$ is transitive.
\end{lemma}

\begin{proof}
    Let $p,q,r\in\mathbb M^\ast$ such that $r\le q$ and $q\le p$.
    By the definition of $\le$ we get that 
    \[ r_\emptyset =\langle A^r\cup \mathbf a(x(r))\big\rangle \Vdash \dot c^q \trle \dot c ^r \text{ and}\]
    \[ q_\emptyset =\langle A^q\cup \mathbf a(x(q))\big\rangle \Vdash \dot c^p \trle \dot c ^q.\]
    Notice that $\mathbf a(x(r))\s \mathbf a (x(q))$, thus 
    \[ r_\emptyset =\big\langle A^r\cup \mathbf a(x(r))\big\rangle\le_{\mathbb M_{\vec{\mathcal U}}} \big\langle A^q\cup \mathbf a(x(q))\big\rangle =q_\emptyset.\]
    Therefore 
    \[ r_\emptyset=\langle A^r\cup \mathbf a(x(r))\big\rangle \Vdash \dot c^p \trle \dot c ^r.\]
\end{proof}

\subsubsection{The essential projection}\label{subsubsec: proj}
In the proof of Main Theorem \ref{1st maintheorem}, the ancestral component is crucial: it allows us to project onto the relevant $\kappa^+$-closed forcing and use reflection in its generic extension to establish reflection of a stationary subset of $\kappa^+$ in the $\mathbb M^\ast$-generic extension.

\begin{definition}
    Let $x\in\mathbb Y$, define
\[ \mathbb X_x:=\{\dot c\in\mathbb T(\dot{\mathbb C})\mid \exists A\in U(\kappa) \text{ s.t. }\langle x,A,\dot c\rangle\in\mathbb M^\ast\}.\]

For $c,d\in\mathbb X_x$ we say that $d\le c$ if and only if there is some $A\in U(\kappa)$ such that $\langle x,A,d\rangle\le_{\mathbb M^\ast}\langle x,A,c\rangle$.
\end{definition}

\begin{lemma}
    There is a projection $\pi_x\colon\nicefrac{\mathbb M^\ast}{\langle x, \kappa\setminus\max(\dom (x)), \check\emptyset\rangle}\ra \mathbb X_x/\sim$.
\end{lemma}

\begin{proof}
    Let $p\le \langle x, \kappa\setminus\max(\dom(x)),\check\emptyset\rangle$, and define $\pi_x(p):= \dot c^p$.
    Then $\pi_x$ is a projection:
    Let $d\le_{\mathbb X_x} \pi_x(p)$ Thus there is $B\in U(\kappa)$ for which
    \[ \langle B\cup \mathbf a (x)\rangle\Vdash \dot c^p \trle \dot d.\]
    Let $\dot e:= \dot c^p \bigcup \dot d\restriction \langle\emptyset, (B\cap A^p)\cup \mathbf a(x)\rangle$. 
    
    Let $q:= \langle x(p), A^p\cap B, \dot e\rangle$; then clearly $\pi_x(q)=d$ and $q\le^\ast p$.
\end{proof}

\begin{definition}
    For $x\in\mathbb Y$, define:
\[\mathbb D_x:=\big\{\dot c\in \mathbb X_x \big| \exists A \exists \delta<\kappa^+ \langle x, A,\dot c\rangle \in \mathbb M^\ast \text{ and }
\langle A\cup \mathbf a (x) \rangle\Vdash \max(\dot c)=\check \delta \big\}.\]
\end{definition}

\begin{lemma}
    $\mathbb D_x$ is dense in $\mathbb X_x$.
\end{lemma}

\begin{proof}
    Let $\dot c\in \mathbb X_x$, then there is $A\in U(\kappa)$ such that 
    $p:=\langle x,A,\dot c\rangle\in\mathbb M^\ast$, i.e., $\langle A\cup\mathbf a(x)\rangle\Vdash \dot c\in\dot{\mathbb C}$.
    Let
    \[\delta:=\sup\{\gamma<\kappa^+\mid\exists r\le \langle A\cup\mathbf a(x) \rangle (r\Vdash \max \dot c(y)=\check\gamma\})+1.\]

    Since $\mathbb M$ has the $\kappa^+$-chain condition, $\delta<\kappa^+$.
    Let us define $\dot d:= \dot c\cup\{\check \delta\}$.

    Then it holds that $\langle A\cup\mathbf a(x)\rangle\Vdash \dot c\trle \dot d$ and $\langle A\cup\mathbf a(x) \rangle\Vdash \max\dot d=\check \delta$. 
    
    Hence $\langle x, A, \dot d\rangle \le \langle x, A, \dot c\rangle$ and therefore  $\dot d\in\mathbb D_x$ and $\dot d\le_{\mathbb X_x}\dot c$.
\end{proof}

\begin{lemma}
    $\mathbb D_x$ is $\kappa^+$-closed.
\end{lemma}

\begin{proof}
    Let $\langle c^i\mid \max(x)<i<\kappa\rangle$ be a decreasing sequence in $\mathbb D_x$.
    Let $i<j<\kappa$. 
    Since $c^j\le_{\mathbb X_x} c^i$, there is $A^{j,i}$ such that $\langle x, A^{j,i}, c^j\rangle \le \langle x, A^{j,i}, c^i\rangle$.
    For $i<\kappa$, since $c^i\in \mathbb D_x$, there are $\delta^i<\kappa^+$ and $A^i\in U(\kappa)$ such that $\langle x, A^i, c^i\rangle\in \mathbb M^\ast$ and it holds that $\langle x, A^i\cup \mathbf a(x)\rangle \Vdash \max (\dot c^i)=\check \delta^i$. Let $j<\kappa$ define $\eta_j :=\sup_{i<j}\delta^i$.
    
    Define $\delta:= (\sup\{\delta^i\mid i<\kappa\})+1$.
    Let $\max (x)=\zeta$.
    For all $\zeta <j<\kappa$ define $B^j:= \underset{\iota<j}{\bigcap} A^{j,\iota}\cap A^j$ and define $B:= \underset{j<\kappa}{\triangle} B^j$.
    
    We want to define $\dot c\in \mathbb T(\dot{\mathbb C})$. 
    Let $\zeta<j<\kappa$ and define 
    \[\dot b_0^j:=\{\langle\check\gamma,p\rangle\mid \exists \zeta<i<j, p\le \langle \kappa\setminus \zeta\cup \mathbf a(x)\rangle \text{ with } j\in\dom (p) \text{ such that } p\Vdash \check\gamma\in c^i\},\]
    and let
    \[ \dot b_1^j:=\{\langle\check\eta_j,p\rangle\mid p\le \langle \kappa\setminus \zeta\cup \mathbf a(x)\rangle \text{ with } j\in\dom (p)\},\]
    finally take $\dot b^j=\dot b^j_0\cup\dot b^j_1$.
    Now we can define
    \[ \dot c:= \{\check{\delta}\}\bigcup_{\zeta<j<\kappa} \dot b^j .\]
    Let $\zeta<i<\kappa$. We want to show that $c\le_{\mathbb X_x} c^i$. 
    We claim that 
    $\langle (B\setminus i)\cup \mathbf a(x)\rangle\Vdash \dot c^i\trle \dot c$: \\
    Let $p\le\langle (B\setminus i)\cup \mathbf a(x)\rangle$. Then there is some $\rho\in B\setminus i$ and $q\le p$ such that $\rho\in\dom(x(q))$.
    Then by the definition of $\dot c$, $q\Vdash \dot c=\{\check\delta^\rho,\check\delta\}\cup \bigcup_{j<\rho}\dot c^j$.
    Let $\zeta<i<j<\rho$. Then $B\setminus \rho\s B^j$ and clearly $q\le\langle A^{j}\cup \mathbf a(x)\rangle\le\langle A^{j,i}\cup \mathbf a(x)\rangle$, thus $q\Vdash \dot c^i\trle \dot c^j$.
    This implies that $q\Vdash \dot c^i\trle \dot c$ \footnote{Notice that for all $l<i$, $\langle (B\setminus i)\cup \mathbf a(q(x))\rangle\Vdash \dot c^l\trle \dot c^i$ because $\langle x(q), (B\setminus i)\cup \mathbf a(x)\rangle\le \langle A^i\cup \mathbf a(x)\rangle$.} 
    and thus $\langle (B\setminus i)\cup \mathbf a(x)\rangle \Vdash \dot c^i\trle \dot c$.
 
    Hence for every $\zeta<i<\kappa$ $\langle x,B\setminus i,c\rangle\le \langle x, B, c^i\rangle$ and therefore $\dot c\le_{\mathbb X_x} \dot c^i$.
\end{proof}

\subsubsection{\texorpdfstring{The Strong Prikry Property and the preservation of $\kappa^+$}{ Prikry Property}}\label{sec:prikryproperty}

\begin{definition}
We call a tree $T\s[\alpha]^{<\omega}$ \emph{$\alpha$-fat} if there is $n<\omega$ such that $h(T)=n$ and for any non-maximal $t\in T$ there is some $i<o^{\mathcal {U}}(\alpha)$ with $\Succ_T(t)\in U_{\alpha,i}$.

For some condition $p\in\mathbb M^\ast$ and an $\alpha$-fat tree $T$ we say that $T$ is \emph{compatible} with $p$ if for any maximal $t\in T$, $t\in \big(\mathbf{a}(x(p))\cup A^p\big)\cap\alpha$ and either $\alpha=\kappa$ or $\alpha\in \dom (p)$.
\end{definition}

\begin{definition}  Let $p\in \mathbb M^\ast$.
\begin{enumerate}
    \item Let $\gamma\in \mathbf a(x(p))\cup A^p$. We define $p+\langle\gamma\rangle$ as follows:
    \begin{itemize}
        \item If $\gamma\in A^p$ then we take $p+\langle\gamma\rangle:=\langle x(p),(\gamma, A^p\downarrow\gamma), A^p\setminus (\gamma+1),\dot c^p\rangle$.
        \item Otherwise, there is some $k<\ell(p)$ such that $\gamma\in A^p_k$ and we take\\
        $p+\langle \gamma\rangle:= \langle (\alpha^p_0,A^p_0),\ldots\allowbreak,(\alpha^p_{k-1},A^p_{k-1}), (\gamma, A_k^p\downarrow\gamma),(\alpha^p_{k},A^p_{k}\setminus\gamma+1),\break \ldots,(\alpha^p_{\ell(p)-1},A^p_{\ell(p)-1}),\dot c^p\rangle$.
    \end{itemize}
    Where $A\downarrow \gamma:=
    \begin{cases}
        \emptyset & o^{\mathcal {U}}(\gamma)=0,\\
        \{\beta\in A\cap \gamma\mid o^{\mathcal {U}}(\beta)<o^{\mathcal {U}}(\gamma) \} & otherwise.
    \end{cases}$
    \item Let $t=\langle\gamma_0,\ldots,\gamma_{n-1}\rangle$ such that for all $k<n$, $\gamma_k\in \mathbf a(x(p))\cup A^p$.
    We define $p+t:= p +\langle\gamma_0\rangle +\dots +\langle\gamma_{n-1}\rangle$.
\end{enumerate}
\end{definition}

\begin{theorem}[The Strong Prikry Property of $\mathbb M^\ast$] \label{Strong Prikry Property}
    Let $p\in\mathbb M^\ast$ and let $D\s \mathbb M^\ast$ open dense. 
    Then there are $p^\ast\le^\ast p$ and trees $T_{0},\ldots,T_{\ell(p)}$, where $T_k$ is $\alpha^p_k$-fat for $k < \ell(p)$ and $T_{\ell(p)}$ is $\kappa$-fat, such that for every $k\le \ell(p)$ and any maximal branch $t_k\in T_k$
    $p^\ast+ t_0 +\dots + t_{\ell(p)}\in D$.
\end{theorem}

\begin{definition}
    Let $p\in\mathbb M^\ast$ and $D\s\mathbb M^\ast$ open dense.
    We define $B(p,D):=\{\beta\in A^p\mid  p+\langle \beta\rangle\in D\}$.
\end{definition}

\begin{lemma}\label{Large Dichotomy for SPP} Let $q\in\mathbb M^\ast$ and $D\s \mathbb M^\ast$ open dense. Then there is $q'\le^\ast q$ such that one of the following holds:
    \begin{enumerate}
        \item There is $i<o^{\mathcal {U}}(\kappa)$ with $B(q',D)\in U_{\kappa,i}$.
        \item For all $\beta\in A^{q'}$ there is no $r\le^\ast q'+\langle\beta\rangle$ with $r\in D$.
    \end{enumerate}
\end{lemma}

\begin{proof}[Proof of the Lemma]
Let $p_0=\langle x(q), A^q\setminus B(q,D),c^q\rangle$.
If there is some $i$ with $B(p_0,D)\in U_{\kappa,i}$ then $(1)$ holds for $p_0$.
Otherwise let \[B_0:=\{\beta\in A^{p_0}\setminus B(p_0,D)\mid\exists q_\beta\le^\ast q+\langle \beta\rangle\wedge q_\beta\in D\}.\]
If $B_0\neq \emptyset$ take $\beta_0=\min(B_0)$ and fix $q_{\beta_0}\le^\ast q+\langle \beta_0\rangle$ with $q_{\beta_0}\in D$.
Let $p_1:=\langle x(q),A^{p_0}, c^{q_{\beta_0}}\rangle$. 
Notice that $p_1+\langle\beta_0,A^{q_{\beta_0}}_{\beta_0}\rangle=q_{\beta_0}\in D$.

Let us define $p_\alpha,B_\alpha,\beta_\alpha,q_\alpha$ by induction on $\alpha<\kappa$: 
Assume that for all $\gamma<\alpha$ $p_\gamma,B_\gamma,\beta_\gamma$ and $q_{\beta_\gamma}$ were fixed and defined.
For $\alpha=\gamma+1$ successor let $p_\alpha:=\langle x(q),A^{q_{\beta_{\gamma}}},c^{q_{\beta_{\gamma}}}\rangle$.
Let 
\[B_{\alpha}:=\{\xi\in A^{p_\alpha}\setminus B(p_\alpha,D)\mid\exists q_\xi\le^\ast (p_\alpha+\langle \xi\rangle) \wedge q_\xi\in D\}.\]
For $\alpha$ limit let $p_\alpha\le^\ast\big\langle \langle x(q), A^{p_\gamma}, c^{p_\gamma}\rangle\big|\gamma<\alpha\big\rangle$ and let $B_\alpha=\underset{\gamma<\alpha}{\bigcap}B_\gamma$.
If $B_\alpha\neq\emptyset$ let $\beta_{\alpha}=\min(B_{\alpha})$ and fix $q_{\beta_\alpha}\le^\ast p_\alpha+\langle \beta_{\alpha}\rangle$ with $q_{\beta_{\alpha}}\in D$.

If there is some $\alpha$ such that $\kappa\setminus B_\alpha\in U(\kappa)$ then either (1) or (2) holds for $p_\alpha$.

Let $A':=\triangle_{\alpha<\kappa} A^{p_\alpha}$, $B':=\triangle B_\alpha$, $c'\le\langle c^{p_{\alpha}}\mid \alpha<\kappa\rangle$ and let $p'=\langle x(q),A',c'\rangle$.

If there is some $i$ such that $B'\in U_{\kappa,i}$,
then fix some $x$ such that $S_x=\{ \xi<\kappa\mid x(q_{\beta_\xi}\restriction\beta_\xi)=x\}\in U_{\kappa,i}$.
Notice that since $\alpha\mapsto\beta_\alpha$ is continuous we can take $B_i=\{\xi\in B'\cap S_x\mid \beta_\xi=\xi\}\in U_{\kappa,i}$.
Let $A_i^0=[\alpha\mapsto A^{q_\alpha}_\alpha]_{U_{\kappa,i}}$.
Notice that $A_i^0\in \bigcap_{j<i} U_{\kappa,j}$.
Let $A^1_i=\{\xi\in A'\cap B_i\mid A_i^0\cap\xi\in U(\xi) \}$, 
$A^1_i\in U_{\kappa,i}$. Let $A_i^2:=\{\xi\in A\cap B_i\mid A_i^1\cap \xi\in U_{\xi,i}\}\in \cap_{j>i} U_{\kappa,j}$.

Let $A''=A'\cap \big( A_i^0\cup A_i^1\cup A_i^2\big)\in U(\kappa)$ and let $p''=\langle x, A'',c'\rangle$. Then for all $\beta\in B=A^1_i$
$p''+\langle\beta\rangle\le^\ast q_\beta\in D$.

If there is no $i$ such that $B'\in U_{\kappa,i}$,
then $\{\xi\mid \exists r'\le p''+\langle\xi\rangle \text{ s.t. } r'\in D\}\s B'$. Take $q''=\langle x(q), A'\setminus B', c'\rangle$ and then $(2)$ holds for $q''$.
\end{proof}

\begin{proof}[Proof of Theorem \ref{Strong Prikry Property} (The Strong Prikry Property of $\mathbb M^\ast$)]
If there is $p^\ast\le^\ast p$ with $p^\ast\in D$, we are done.
Without loss of generality, assume $p=\langle A,c\rangle$.
    
By Lemma~\ref{Large Dichotomy for SPP} there is
$p_0\le^\ast p$ such that either 
\begin{enumerate}
    \item there is $i<o^{\mathcal {U}}(\kappa)$ with
         $B(p_0,D)\in U_{\kappa,i}$, or
    \item for all $\beta\in A^{p_0}$ there is no $q\le^\ast p_0 + \langle\beta\rangle$ such that $q\in D$.
\end{enumerate}
Let $\langle t_\alpha \mid \alpha<\kappa \rangle$ be an enumeration of $\mathbb Y$ such that if there are $X_\alpha,X_\beta\in U(\kappa)$ such that $\langle t_\alpha, X_\alpha\rangle\le_{\mathbb M_{\vec{\mathcal U}}} \langle t_\beta,X_\beta\rangle$
then $\beta<\alpha$.

We define $p_{\alpha}$ by induction on $\alpha<\kappa$:
Assume that for some $\alpha<\kappa$ for all $\gamma<\alpha$, $p_\gamma$ was defined.
Use Lemma \ref{Large Dichotomy for SPP} for $p_\alpha' \le^\ast \langle p_\gamma\mid \gamma<\alpha\rangle$ and $D$ and fix $p_\alpha\le^\ast p_\alpha'$ such that 
\begin{enumerate}
    \item either there is $i_\alpha <o^{\mathcal {U}}(\kappa)$ with $B(p_\alpha +t_\alpha,D)\in U_{\kappa,i_\alpha}$, or
    \item for all $\beta\in A^{p_\alpha}$  there is no $q\le^\ast p_\alpha +t_\alpha + \langle\beta\rangle$ such that $q\in D$.
\end{enumerate}

Let $A^0:=\triangle_{\alpha<\kappa} A^{p_\alpha}$, $c^0\le\langle c^{p_\alpha}\mid \alpha<\kappa\rangle$.
Define $q_0:= \langle A^0,c^0\rangle$.
Let $D^0:= D$. For $k<\omega$, define 
\[ D^{k+1}:=\{r\in\mathbb M^\ast\mid \exists i<o^{\mathcal {U}}(\kappa) \text{ s.t. } B(r,D^k)\in U_{\kappa,i}\}.\]
Suppose that $q^k$ has been defined for some $k<\omega$.
We define $p^{k+1}_{\alpha}$ by induction on $\alpha<\kappa$:
Assume that for some $\alpha<\kappa$,
for all $\gamma<\alpha$, $p^{k+1}_\gamma$ was defined.
Use Lemma \ref{Large Dichotomy for SPP} to get $p^{k+1}_\alpha \le^\ast \langle p_\gamma^{k+1}\mid \gamma<\alpha\rangle$ such that either 
\begin{enumerate}
    \item there is $i^{k+1}_\alpha <o^{\mathcal {U}}(\kappa)$ with $B(p^{k+1}_\alpha+t_\alpha,D^{k+1})\in U_{\kappa,i^{k+1}_\alpha}$, or
    \item for all $\beta\in A^{p^{k+1}_\alpha}$  there is no $q\le^\ast p^{k+1}_\alpha +t_\alpha + \langle\beta\rangle$ such that $q\in D$.
\end{enumerate}
Let $A^{k+1}:=\triangle_{\alpha<\kappa} A^{p^{k+1}_\alpha}$, $c^{k+1}\le\langle c^{p^{k+1}_\alpha}\mid \alpha<\kappa\rangle$.

Finally let $A^\ast=\bigcap A^n$ and $c^\ast\le\langle c^n\mid n<\omega\rangle$.
Notice that if $q\le \langle A^\ast,c^\ast\rangle$ is such that $q\in D$ then there is a $\kappa$-fat tree $T$ with $h(T)=\ell(q)$ such that for all maximal $t\in T$  $p^\ast+t  \in D$.
\end{proof}

\begin{lemma}
    $\mathbb M^\ast$ preserves $\kappa^+$. 
\end{lemma}

\begin{proof}
    Let $p\in\mathbb M^\ast$ and $\mu<\kappa$ and let $\dot f$ a name such that $p\Vdash \dot f\colon \check \mu\ra \check{\kappa^+}$. Assume without loss of generality that $p=\langle A, c\rangle$.
    For $\xi<\mu$, define \[D_\xi:=\{q\in \mathbb M^\ast\mid q\Vdash \dot f(\check\xi)<\max (\dot c^q)\}.\] 
    Clearly $D_\xi$ is dense in $\mathbb M^\ast$.
    Now use the Strong Prikry Property (Theorem \ref{Strong Prikry Property}), and obtain $p_0\le^\ast p$ and and a $\kappa$-fat tree $T_0$ such that for all maximal $t\in T_0$ $p_0+t\in D_0$. 
    Thus $p_0\Vdash \dot f(0)<\max (\dot c^{p_0})$.

    Suppose that, for some $\xi<\mu$, the conditions $p_\zeta$ have been defined for all $\zeta<\xi$ and $\langle p_\zeta\mid \zeta<\xi\rangle$ is a decreasing sequence. 
    Let $q_\xi\le^\ast \langle p_\zeta\mid \zeta<\xi\rangle$. Now apply the Strong Prikry Property for $D_\xi$ and $q_\xi$ to obtain $p_\xi$ and $T_\xi$.
    Notice that $p_\xi\Vdash \dot f(\xi)<\max(\dot c^{p_\xi})$.
    Finally, let $p^\ast \le^\ast \langle p_\xi \mid \xi<\mu\rangle$ such that there is some $\delta<\kappa^+$ with $\langle A^{p^\ast}\rangle\Vdash \max(\dot c^{p^\ast})=\check \delta$.
    Thus $p^\ast\Vdash\dot f\colon\check\mu\ra\check\delta$.
\end{proof}

\begin{lemma}\label{bounded new subsets below kappa are from M}
    Let $p\in\mathbb M^\ast$, $\gamma\in \dom(p)$ with $o^{\mathcal {U}}(\gamma)>0$, and $\dot x$ such that $p\Vdash \dot x\s \check\gamma$. 
    Then there is $p'\le^\ast p$ with $p'\restriction_\gamma=p\restriction_\gamma$ and there is an $\mathbb M\downarrow \gamma$ name $\tilde y$ such that $p'\Vdash \dot x=\tilde y$.
\end{lemma}
\begin{proof}
    Let $\langle (\xi_i,r_i)\mid i<\gamma\rangle$ be an enumeration of $\gamma\times [\gamma]^{<\omega}$.

    For convenience let us define for all $x\in \mathbb Y$: $x\downarrow\gamma:=\langle (\alpha, A^x_\alpha)\mid \alpha\in \dom(x)\cap \gamma+1\rangle$ and $x\uparrow\gamma:=\langle (\alpha, A^x_\alpha)\mid \alpha\in \dom(x)\setminus\gamma+1\rangle$.
    
    We build $s_i$ and $p_i$ inductively on $i<\gamma$:\\
    $\br$ Base case: for $i=0$, from the Prikry Property for $\mathbb M^\ast$ there is some $q\le p+r_{0}$ such that $q\| \check\xi_0\in \dot x$.
    Let $s_0:=x(q^0)\downarrow\gamma$, and let $p_0:=\langle x(p)\downarrow\gamma, x(q^0)\uparrow\gamma, A^q, c^q\rangle$.\\
    $\br$ Successor step: Suppose that $i=j+1$ and that $s_j$ and $p_j$ were defined. 
    Use the Prikry Property for $\mathbb M^\ast$ to get $q_i\le^\ast p_j+ r_i$ such that $q_i\|\check\xi_i\in \dot x$.
    Let $s_i:= x(q_i^0)\downarrow\gamma$ and let 
    $p_i:= \langle x(p)\downarrow\gamma, x(q_i^0)\uparrow\gamma, A^{q_i},c^{q_i}\rangle$.\\
    $\br$ Limit step: Let $i$ be a limit ordinal, and suppose that for all $j<i$ $s_j$ and $p_j$ were defined.
    
    Let $p'_i\le^\ast \langle p_j\mid j<i\rangle$. 
    Use the Prikry Property to get $q_i\le^\ast p'_i +r_i$ such that $q_i\|\check\xi_i\in \dot x$.
    Let $s_i:= x(q_i^0)\downarrow\gamma$ and let 
    $p_i:= \langle x(p)\downarrow\gamma, x(q_i^0)\uparrow\gamma, A^{q_i},c^{q_i}\rangle$.

    Now we can take $p_\gamma\le^\ast \langle p_j\mid j<\gamma\rangle$.
    Let 
    $\tilde y:=\{(s,\check\alpha)\mid s\in\mathbb M\downarrow \gamma/(x(p)\downarrow\gamma)\wedge p_\gamma+s \Vdash \check\alpha\in\dot x\}$.
    
    \begin{claim}
        $p_\gamma \Vdash \dot x=\tilde y$.
    \end{claim}
    \begin{proof}
        Let $\alpha<\gamma$. 
        Let $q\le p_\gamma$ such that $q\|\check \alpha \in \dot x$, then there is some $i<\gamma$ such that $(\dom(q^0)\cap\gamma,\alpha)=(r_i,\xi_i)$. 
        Notice that since $p_\gamma+ s_i$ is compatible with $q$, $q\Vdash\check \alpha \in \dot x$ iff $p_\gamma+ s_i\Vdash\check\alpha\in\dot x$ iff $s_i\Vdash_{\mathbb M\downarrow \gamma} \check\alpha \in \tilde y$.

        Therefore $p_\gamma\Vdash \dot x=\tilde y$.
    \end{proof}
\end{proof}

\begin{corollary}
    $\mathbb M^\ast$ preserves all cardinals.
    Moreover, for all $\gamma<\kappa$ which are not limit points of the generic Magidor club, all stationary subsets of $\gamma$ are preserved by $\mathbb M^\ast$. \qed
\end{corollary}

\subsection{\texorpdfstring{Constructing $\mathbb R^\ast$}{The forcing R*}}\label{subsec: R*}
The idea is to combine the forcing $\mathbb M^\ast$ from the previous subsection with a different forcing $\mathbb R$ that makes $\kappa=\aleph_\nu$, where $\nu<\kappa$ and $\cf(\nu)=\rho$. 
Both forcings project onto a Magidor forcing.

The following definition and ideas appear in \cite{AMGC}.
\begin{definition}[Guru]\label{def_guru} For a regular cardinal $\alpha$, a sequence $\vec{t}=\langle t_i \mid i<\alpha^{+} \rangle$  is a \emph{guru for $\alpha$} iff all of the following hold:
\begin{enumerate}
    \item for every $i<\alpha^{+}$, $t_i:\alpha\rightarrow V$ is a function such that $t_i(\beta) \in \col (\beta^+, <\alpha)$ for every regular $\beta<\alpha$;
    \item for all $j<i<\alpha^{+}$, the set $\{ \beta<\alpha \mid t_j(\beta) \nsubseteq t_i(\beta)\}$ is non-stationary in $\alpha$;
    \item for every function $D: \alpha \to V$ such that $D(\beta)$ is a dense subset of $\col(\beta^+,<\alpha)$ for every $\beta<\alpha$, 
   		 there is an $i<\alpha^{+}$ such that $t_i(\beta) \in D(\beta)$ for all $\beta<\alpha$.
\end{enumerate}
\end{definition}

The following lemmas regarding a guru are derived from \cite[Lemma 3.3]{AMGC} the difference between the gurus used here and those used there is that a guru at $\alpha$ in \cite{AMGC} is a sequence of $\alpha^{++}$ functions such that each is a function with domain $\alpha$ and maps $\beta<\alpha$ to an element of $\col(\beta^{++},<\alpha)$, whereas here a guru at $\alpha$ has length $\alpha^{+}$ and each function maps $\beta<\alpha$ to an element of $\col(\beta^+,<\alpha)$. 
Notice that there is no essential difference in the proofs.
\begin{lemma}
\label{guru_lemma}
Suppose that $\alpha$ is a regular limit cardinal such that $2^{\alpha}=\alpha^{+}$. Then there is a guru for $\alpha$.\qed
\end{lemma}
Similarly, using this version of the notion of a guru we have the following lemma.

\begin{lemma}\label{coherentsequence}
Let $\nu$ be some limit ordinal with $\cf(\nu)=\rho$.
Let $\kappa$ be such that $o(\kappa)\geq \nu$, then there are:
    \begin{itemize}
    \item a map $o^{\mathcal {U}}$ from a set of strongly inaccessible cardinals to the ordinals,
     \item a sequence $\vec{U}=\left\langle\langle U_{\alpha,i} \mid i<o^{\mathcal {U}}(\alpha)\rangle\Mid \alpha \in \dom(o^{\mathcal {U}})\right\rangle$, and
    \item a sequence $\vec{\mathbf{t}}=\langle \vec{t}_\alpha \mid \alpha \in \dom(o^{\mathcal {U}}) \rangle$ 
    \end{itemize}
such that all of the following hold:
\begin{enumerate}[label=\textup{(\arabic*)}]
    \item  $\max(\dom(o^{\mathcal {U}}))=\kappa$ with $o^{\mathcal {U}}(\kappa)=\nu$;
    \item for every $\alpha \in \dom(o^{\mathcal {U}})$, $\vec{t}_\alpha$ is a guru for $\alpha$;
    \item for every $\alpha \in \dom(o^{\mathcal {U}})$ and every $i<o^{\mathcal {U}}(\alpha)$:
    \begin{enumerate}[label=\textup{(\alph*)}]
    \item $U_{\alpha,i}$ is a normal measure on $\alpha$,
    and we let $j_{\alpha,i}: V \to M_{\alpha,i}$ denote the corresponding ultrapower embedding;
    \item $j_{\alpha,i}(\vec{U}\restriction \alpha) \restriction \alpha=\vec{U} \restriction \alpha$;
    \item $j_{\alpha,i}(\vec{U}\restriction \alpha) (\alpha)=\langle U_{\alpha,k}\mid k<i\rangle$;
    \item $j_{\alpha,i}(\vec{\mathbf{t}} \restriction \alpha)(\alpha)=\vec{t}_\alpha$.
\end{enumerate}
\end{enumerate} 
\qed
\end{lemma}

The following forcing notion is an alteration of the forcing notion with gurus that appears in \cite{AMGC} where the key difference is the collapses and gurus.
\begin{definition}[\cite{AMGC}]
    Let the forcing notion $\mathbb R:=\mathbb{P}_{\vec{U},\vec{\mathbf{t}}}$ be defined as the set of conditions of the form:
    $p:=\langle e_0,(\alpha_0, A_0, i_0),\ldots,e_{n-1},(\alpha_{n-1}, A_{n-1},i_{n-1}), e_n,(\kappa, A_n, i_n)\rangle$, where
    \begin{enumerate}
        \item  $i_n<\kappa^+$ and $A_n\in\bigcap U(\kappa)$,
        \item For $k< n$: \\
        $\br$ if $o^{\mathcal {U}}(\alpha_k)\neq0$ then $i_k<\alpha_k^+$ and $A_k\in\bigcap U(\alpha_k)$,\\
        $\br$ otherwise $i_k=0$ and $A_k=\emptyset$.
        \item \begin{enumerate}
            \item $e_0\in\col(\nu^+,<\alpha_0)$,
            \item for all $1\le k\le n-1$ $e_k\in\col(\alpha_{k-1}^+,\alpha_k)$,
            \item $e_n\in\col(\alpha_{n-1}^+,<\kappa)$.
        \end{enumerate}
    \end{enumerate}
\end{definition}
Let $p,q\in \mathbb R$. We say that $q\le p$ iff
\begin{enumerate}
    \item $\dom(p)\s\dom(q)$.
    \item Let $\beta_0=\min(\dom(q))$, then $e_0^q\le_{\col(\omega_1,<\beta_0)} e_0^p$. 

    \item For all $\alpha\in\dom(p)$, let $\beta_\alpha=\min(\dom(q)\setminus\alpha+1)$, then $e_\alpha^q\le_{\col(\alpha^+,<\beta_\alpha)} e_\alpha^p$, 
    and $A_\alpha^q\s A_\alpha^p$ and $i_\alpha^p\le i_\alpha^q$.
    \item For all $\beta\in \dom (q)\setminus \dom (p)$, let $\alpha(\beta)\min \{\kappa\}\cup\dom(p)\setminus \beta+1$, then 
    \begin{itemize}
        \item $\beta\in A^p_{\alpha(\beta)}$,
        \item if $o^{\mathcal {U}}(\beta)>0$ then $A_\beta^q\s A^p_{\alpha(\beta)}$, 
        \item $e^q_\beta\le_{\col(\beta^+, \min(\dom(q)\setminus\beta+1))} t_{\alpha(\beta),i^p_{\alpha(\beta)}} (\beta)$.
    \end{itemize}
\end{enumerate}

\begin{remark}
By \cite{AMGC} all of the following statements hold for the forcing $\mathbb R:=\mathbb{P}_{\vec{U},\vec{\mathbf{t}}}$
    \begin{enumerate}
        \item $\mathbb R$ has the $\kappa^+$-chain condition.
        \item $\mathbb R$ has the Strong Prikry Property:
        Let $r\in \mathbb R$ with $\dom (r)=\{\alpha_0,\ldots,\alpha_{n-1}\}$ and $D\s \mathbb R$ dense open, then there exists $r'\le^\ast r$ and $T_0,\ldots, T_{n}$ trees that are $\alpha_0,\ldots,\alpha_{n-1},\kappa$-fat such that for every maximal $t_k\in T_k$ there is $B_{t_k}=\langle B_{t_k}^0,\ldots,B_{t_k}^{\ell(t_k)-1}\rangle$, where for each $l<\ell(t_k)$ $B_{t_k}^{l}$ is of $t_k(l)$ measure one, such that 
        \[ r'+ \langle t_0 ,\vec B_{t_0}\rangle +\dots +\langle t_{n} ,\vec B_{t_{n}}\rangle\in D.\]
        \item If $r\in \mathbb R$ and $\alpha\in \dom (r)$ then $\mathbb R/r\cong \mathbb R^{l,r}_\alpha\times \mathbb R^{u,r}_\alpha$ such that $\mathbb R^{l,r}_\alpha$ has the $\alpha$-chain condition and $\langle \mathbb R^{u,r}_\alpha,\le^\ast\rangle$ is $\alpha^+$-closed.
        \item There is a projection $\pi_{1}:\mathbb R\ra \mathbb M$.
        \item If $H'\s \mathbb R$ is generic then in $V[H']$:
        \begin{enumerate}
            \item If $\alpha\in \{\omega,\omega_1\}\cup \{\beta,\beta^+\mid \exists r\in\pi_1(H') ~ \beta\in\dom(rp)\}$ then $\alpha$ is a cardinal in $V[H']$.
            \item If $\alpha\notin\{\omega,\omega_1\}\cup \{\beta,\beta^+\mid \exists r\in\pi_1(H') ~ \beta\in\dom(r)\}$ and $\alpha<\kappa$, then $\alpha$ is not a cardinal in $V[H']$.
            \item $\kappa=\aleph_{\omega_1}$.
        \end{enumerate}
    \end{enumerate}
\end{remark}

\subsubsection*{\texorpdfstring{Constructing $\mathbb R^\ast$:}{R* construction}}\label{subsec: downto construction}

Let $(\vec{\mathcal{U}},\vec t)$ be a coherent measure and guru sequences with $o^{\mathcal {U}}(\kappa)=\nu$ from Lemma~\ref{coherentsequence}, where $\nu<\kappa$ is a limit ordinal and fix $\rho=\cf(\nu)$.
Let $\mathbb R=\mathbb P_{\vec{\mathcal{U}},\vec t}$ and $\mathbb M=\mathbb M_{\vec{\mathcal{U}}}$ (the Magidor forcing with the coherent sequence $\vec{\mathcal{U}}$) and $\mathbb M^\ast$ as defined previously.

Fix $\pi_{0,0}\colon\mathbb R\rightarrow\mathbb M$ and
$\pi_{1,0}\colon\mathbb M^\ast\rightarrow\mathbb M$ projections.

\begin{definition}
    The forcing $\mathbb R^\ast$ consists of conditions of the form 
    \[p=\langle p^0, \dot c\rangle,\] where
    \begin{enumerate}
        \item $p^0\in\mathbb R$.
        \item $p^1:= \langle \pi_{0,0}(p^0), \dot c\rangle\in \mathbb M^\ast$.
    \end{enumerate}
\end{definition}

\begin{definition}
    Let $q,p\in\mathbb R^\ast$. We say that $q\le p$ iff
        $q^0\le_{\mathbb R}p^0$ and $q^1\le_{\mathbb M^\ast} p^1$.
    We say that $q\le^\ast p$ iff $q\le p$ and $q^0\le^\ast_{\mathbb R}p^0$.
\end{definition}

\begin{lemma}
    There is a projection $\pi^1:\mathbb R^\ast\ra\mathbb M^\ast$ defined by $p\mapsto p^1$. 
\end{lemma}

Let \[\mathbb Y':=\{ y\mid \exists A\in U(\kappa),\exists i<\kappa^+,c\in\mathbb T(\dot{\mathbb C}) ~\langle y,(A,i),c\rangle\in \mathbb R^\ast\}.\] Notice that $|\mathbb Y'|=\kappa$.
For every $y\in \mathbb Y'$
let $\mathbb Z_y:=\mathbb X_{y^1}$
and set $c'\le_{\mathbb Z_y} c$ iff there are $A\in U(\kappa)$ and $i<\kappa^+$ such that $\langle y,(A,i),c'\rangle\le \langle y,(A,i),c\rangle$.
\subsubsection{Strong Prikry Property}\label{sec:prikrypropertydown}
\begin{theorem} \label{Strong R* Prikry Property} 
Let $p\in\mathbb R^\ast$ and $D\s \mathbb R^\ast$ dense open.
    Then there exists a condition $p'\le^\ast p$ and trees $T_0,\ldots, T_{\ell(p)}$ that are $\alpha_0,\ldots,\alpha_{{\ell(p)}-1},\kappa$-fat such that for every maximal $t_k\in T_k$ there is $B_{t_k}$ of $t_k$ measure one such that 
    \[ p'+ \langle t_0 ,\vec B_{t_0}\rangle +\dots +\langle t_{\ell(p)} ,\vec B_{t_{\ell(p)}}\rangle\in D.\]
\end{theorem}

\begin{proof} Without loss of generality, assume that $p=\langle e, (A,i), c\rangle\in\mathbb R^\ast$.
Let $c'\in\mathbb T(\dot{\mathbb C})$, define $D_{c'}:=\{r\in\mathbb R\mid \langle r,c'\rangle\in D\}$.
Let $D^1:=\{q\in\mathbb M^\ast \mid \exists r\in\mathbb R (\pi_{0,0}(r)=\pi^0(q))\wedge( D_{c^q} \text{ is dense below } r)\}$.

\begin{claim}
    $D^1$ is dense in $\mathbb M^\ast$.
\end{claim}
\begin{proof}
    Let $q\in\mathbb M^\ast$ and let $p\in \mathbb R^\ast$ such that $\pi^1(p)=q$.
    Let $p'\le p$ be in $D$. Then $q':=\pi^1(p')\le q$. Let $r':= \pi^0 (p')$. Clearly $D_{c^{q'}}$ is dense below $r'$ therefore $q'\in D^1$.
\end{proof}
Since $D^1$ is dense, in $\mathbb M^\ast$ apply the Strong Prikry Property, Theorem~\ref{Strong Prikry Property} for $p^1$ and $D^1$ and obtain $q:=\langle A^\ast, c^\ast\rangle \le^\ast_{\mathbb M^\ast} \pi^1(p)$ and a $\kappa$-fat tree $T$ such that for all maximal $t\in T$ there is $\vec B^1_t$ such that $q+ \langle t,\vec B^1_t\rangle\in D^1$ i.e, there is $r_t\in\mathbb R$ such that $\pi_{0,1}(r_t)=\pi^0(q+\langle t,\vec B^1_t\rangle)$ and $D_{c^\ast}$ is dense below $r_t$.

Let $p^\ast:=\langle e, (A^\ast,i), c^\ast\rangle$ and 
$r^\ast:= \pi^0(p^\ast)$.
Notice that since for all maximal $t\in T$, $r_t\le_{\mathbb R} r^\ast$, $D_{c^\ast}$ is dense below $r^\ast$.

From the Strong Prikry Property for $\mathbb R$, let $r^{**}\le^\ast r^\ast$ and a $\kappa$-fat tree $T^0$ such that for all maximal $s\in T^0$ there is $\vec B^0_s$ such that $r^{**}+\langle s, \vec B_s^0\rangle \in D_{c^\ast}$.

Let $p^{**}:= \langle r^{**}, c^\ast\rangle$ and let $T^\ast$ such that $T^\ast=T\cup \{t+s \mid s\in T^0, t\in T^n\}$.

Thus for every maximal $t'\in T^\ast$ there are $s\in T^0$ and $t\in T^n$ maximal branches such that $t'=t+s$ and set $\vec B_{t'}:= \vec B_t^1+\vec B_s^0$. 
Then for every maximal $t'\in T^\ast$, $p^{**}+ \langle t', \vec B_{t'}\rangle\in D$.
\end{proof}

\begin{lemma} \label{Prikry Property bounding below kappa}
    Let $p\in\mathbb R^\ast$ and $\dot \eta$ such that $p\Vdash \dot \eta<\check\kappa^+$.
    Then there is $\mu<\kappa^+$ and $p^\ast\le^\ast p$ such that $p^\ast \Vdash \dot \eta\le \check\mu$.
\end{lemma}

\begin{proof}~
    Let $D:=\{q\in \mathbb R^\ast \mid q\Vdash \dot \eta <\max (\dot c^q)\}$. Clearly $D$ is dense below $p$.
        Without loss of generality assume $\dom(p)=\emptyset$. 
        Use the Strong Prikry Property \ref{Strong R* Prikry Property} for $p$ and $D$ to obtain $p'\le^\ast p$ and a $\kappa$-fat tree $T$ such that for all maximal $t\in T$ there is $\vec B_t=\langle B_t^0,\ldots,B_t^{\ell(t)-1}\rangle$ a vector of $t$-large sets such that $p'+\langle t, \vec B_t\rangle\in D$.

        Let $p''\le^\ast p'$ and $\mu<\kappa^+$ such that $\langle A^{p''}\rangle\Vdash \max(\dot c^{p''})=\check\mu$.

        Then for all $q\le p''$ there is $t\in T$ compatible with $q$, therefore $q+\langle t, \vec B_t\rangle\in D$.
        Hence $q+\langle t, \vec B_t\rangle\Vdash \dot \eta\le \check \mu$.
        Thus $p''\Vdash \dot \eta\le\check\mu$.
\end{proof}

\begin{lemma}\label{bounded new subsets below kappa are from R}
    Let $p\in\mathbb R^\ast$, $\gamma\in \dom(p)$ and $\dot x$ such that $p\Vdash \dot x\s \check\gamma$. 
    Then there is $p'\le^\ast p$ with $p'\downarrow\gamma=p\downarrow\gamma$ and there is an $\mathbb R\downarrow\gamma$ name $\tilde y$ such that $p'\Vdash \dot x=\tilde y$.
\end{lemma}
\begin{proof} 
Let $r=p\downarrow\gamma$.
Let $\mathbb Y_\gamma=\{y\mid \exists A\in U(\gamma) \text{ s.t. } \langle y,A\rangle\in \mathbb R\downarrow \gamma\}$.
    Let $\langle (\xi_i,r_i)\mid i<\gamma\rangle$ be an enumeration of $\gamma\times (\mathbb Y_\gamma)$.

    We build $s_\alpha,p_\alpha$ and $c_\alpha$ inductively on $\alpha<\gamma$:\\
    $\br$ Base case: For $\alpha=0$, from the Prikry Property for $\mathbb R^\ast$ there is some $q\le^\ast \langle r_{0}, (\gamma,A_\gamma^p, i_\gamma^p),\allowbreak p^0\uparrow\gamma, c^p\rangle$ such that $q\| \check\xi_0\in \dot x$.
    Let $s_0:=q^0\restriction\gamma$, $A_{\gamma,0}:=A^q_\gamma$, $i_{\gamma,0} =i^q_\gamma$, $p_0:= q^0\uparrow^\gamma$ and $c_0:=c^q$.\\
    $\br$ Successor step: Suppose that $\alpha=j+1$ and that $s_j,A_{\gamma,j},i_{\gamma,j}$, $p_j$ and $c_j$ were defined. 
    Use the Prikry Property for $\mathbb R^\ast$ to get $q_\alpha\le^\ast \langle r_\alpha, (\gamma,A_\gamma^p, i_\gamma^p) ,p_j, c_j\rangle$ such that $q_\alpha\|\check\xi_\alpha\in \dot x$.
    
    Let $s_\alpha:=q_\alpha^0\restriction\gamma$, $A_{\gamma,\alpha}:=A^{q_\alpha}_\gamma$, $i_{\gamma,\alpha} :=i^{q_\alpha}_\gamma$,
    $p_\alpha:= q_\alpha^0\uparrow^\gamma$ and $c_\alpha:=c^{q_\alpha}$.\\
    $\br$ Limit step: Let $\alpha$ such that for all $j<\alpha$ $s_j,A_{\gamma,j},i_{\gamma,j},p_j$ and $c_j$ were defined.
    Let $p'_\alpha\le^\ast \Big\langle\langle p^0\downarrow \gamma, p_j, c_j\rangle\Big| j<\alpha\Big\rangle$. 
    Use the Prikry Property for $\mathbb R^\ast$ to get $q_\alpha\le^\ast \langle r_\alpha,(\gamma,A_\gamma^p, i_\gamma^p), p'_\alpha\rangle$ such that $q_\alpha\|\check\xi_\alpha\in \dot x$.
    Let $s_\alpha:=q_\alpha^0\restriction\gamma$, $p_\alpha:= q_\alpha^0\uparrow^\gamma$ and $c_\alpha:=c^{q_\alpha}$.

    Now we can take $p_\gamma\le^\ast_{\mathbb R\uparrow^\gamma} \langle p_j\mid j<\gamma\rangle$ and $c_\gamma\le\langle c_j\mid j<\gamma\rangle$.

    Let $\tilde y:=\{(s,\check\alpha)\mid s\in\mathbb R\downarrow\gamma/r\wedge \langle s, p_\gamma, c_\gamma\rangle\Vdash \check\alpha\in \dot x\}$.
    Let $p'':=\langle r, p_\gamma, c_\gamma\rangle$. Then $p''\Vdash \dot x=\tilde y$.
    
    \begin{claim}
        Let $\alpha<\gamma$, and let $q\le p''$.
        Then $q\Vdash \alpha\in \dot x$ iff $q\downarrow\gamma\Vdash_{\mathbb R\downarrow \gamma} \check\alpha\in\tilde y$.
    \end{claim}
    \begin{proof}
        Let $q\le p''$ such that $q\|\check \alpha \in \dot x$, then there is some $\varepsilon<\gamma$ such that $(q^0\restriction_\gamma,\alpha)=(r_\varepsilon,\xi_\varepsilon)$. 
        Thus there is $s_\varepsilon\le^\ast_{\mathbb R\downarrow\gamma} \langle r_\varepsilon, (\gamma,A_\gamma^p, i_\gamma^p)\rangle$ such that $\langle s_\varepsilon, p_\gamma,c_\gamma\rangle\|\check\alpha\in\dot x$.
        Notice that since $q, \langle s_\varepsilon, p_\gamma,c_\gamma\rangle$ are compatible, we get that they force the same truth value, i.e.,

        \[q\Vdash \check\alpha\in \dot x \Leftrightarrow \langle s_\varepsilon, p_\gamma,c_\gamma\rangle\Vdash \check\alpha\in \dot x.\]       
    \end{proof}
    
    Thus $p''\Vdash \dot x=\tilde y$.
\end{proof}

\begin{theorem}
Let $p\in\mathbb R^\ast$ and let $\delta<\alpha<\kappa$, with $o^u(\alpha)=0$. Suppose that $p\Vdash\dot f\colon\check \delta\ra \check\kappa^+$.
Then there is $p'\le^\ast p$ and some $\mu<\kappa^+$ such that 
$p'\Vdash \dot f:\check \delta\rightarrow\check\mu$.
\end{theorem}

\begin{proof}
    Let $\langle (r_\varepsilon,\xi_\varepsilon)\mid \varepsilon<\alpha\rangle$ be an enumeration of $\big(\mathbb R\restriction_\alpha{/}_{(p\restriction_\alpha)}\big)\times \delta$.

    Similarly to the proof of Lemma \ref{bounded new subsets below kappa are from R} let us construct inductively $s_\varepsilon,p_\varepsilon,c_\varepsilon$ for every $i<\alpha$.
    
    $\br$ Basis case: For $\varepsilon=0$, from Lemma \ref{Prikry Property bounding below kappa} we can obtain some $q_0\le \langle r_0,p\uparrow^\alpha, c^p\rangle$ such that $q_0\Vdash\dot f(\check\xi_0)<\max(\dot c^{q_0})$.
    Let $s_0:=q_0\restriction_\alpha$, $p_0:=q_0^0\uparrow\alpha$ and $c_0:=c^{q_0}$.\\

    $\br$ Successor step: Suppose that $\varepsilon=j+1$ and that $s_j,p_j,c_j$ were defined.
    By using Lemma \ref{Prikry Property bounding below kappa} we can obtain $q_\varepsilon\le^\ast\langle r_\varepsilon, p_j,c_j\rangle$ such that $q_\varepsilon\Vdash\dot f(\check\xi_\varepsilon)<\max(\dot c^{q_\varepsilon})$.
    Let $s_\varepsilon:=q_\varepsilon\restriction_\alpha$, $p_\varepsilon:=q_\varepsilon^0\uparrow\alpha$ and $c_\varepsilon:=c^{q_\varepsilon}$.

    $\br$ Limit step: Suppose that $\varepsilon<\alpha$ is a limit ordinal and that $j<\varepsilon$ $s_j,p_j,c_j$ were defined for all $j<\varepsilon$. Let $q'\le^\ast\Big\langle \langle r_\varepsilon,p_j, c_j\rangle \Big| j<i \Big\rangle$, and by using Lemma \ref{Prikry Property bounding below kappa} obtain $q_\varepsilon\le^\ast q'$  such that $q_\varepsilon\Vdash\dot f(\check\xi_\varepsilon)<\max(\dot c^{q_\varepsilon})$.
    Let $s_\varepsilon:=q_\varepsilon\restriction_\alpha$, $p_\varepsilon:=q_\varepsilon^0\uparrow\alpha$ and $c_\varepsilon:=c^{q_\varepsilon}$.
    
    Let $q_{\alpha}\le^\ast \Big\langle \langle r,p_j, c_j\rangle \Big| j< \alpha\Big\rangle$ and $\mu<\kappa^+$ such that $\langle A^{q_\alpha}\cup\mathbf{a}(q_\alpha)\rangle\Vdash \max(\dot c^{q_\alpha})=\check\mu$.

    \begin{claim}
            $q_\alpha\Vdash \dot f:\check \delta\ra \check\mu$.
    \end{claim}

    \begin{proof}
        Assume on the contrary that there is $q'\le q_\alpha$ and some $\xi<\delta$ such that $q'\Vdash \dot f(\check \xi)\geq \check\mu$.
        Let $\varepsilon<\alpha$ such that $(r_\varepsilon,\xi_\varepsilon)=(q'\restriction_\alpha,\xi)$, then $\langle s_\varepsilon, q'\uparrow^\alpha,c^{q'}\rangle\le q', \langle s_\varepsilon,p_\alpha, c^{p_\alpha}\rangle$. But this is a contradiction because $\langle s_\varepsilon,q_\alpha, c^{q_\alpha}\rangle\le^\ast \langle s_\varepsilon,p_\varepsilon, c_\varepsilon\rangle\Vdash \dot f(\check\xi)<\max(\dot c_\varepsilon)$. 
    \end{proof}
\end{proof}

Let $G\s \mathbb R^\ast$ be a $V$-generic filter and let $H\s\mathbb R$ be the $V$-generic filter derived from the projection.
\begin{corollary}
    $\text{Card}^{V[G]}=\text{Card}^{V[H]}$.\qed
\end{corollary}

\section{Proofs of the main theorems}\label{sec: mainthm proofs}

\begin{theorem}[Main Theorem \ref{1st maintheorem}]\label{main theorem 1}
    Assume $\gch$.
    Let $\rho$ be some regular cardinal.
    Let $\kappa$ be a measurable cardinal with $o(\kappa)\geq \rho$ such that $\rfl(E^{\kappa^+}_{<\kappa})$ holds and is preserved by $\Add(\kappa^+,1)$.
    
    Then, there is a forcing extension in which all cardinals are preserved, $\rfl(\kappa^+)$ holds and $\kappa$ is singular with cofinality $\rho$.  
    
    Moreover, if $\rfl_{<\kappa}(E^{\kappa^+}_{<\kappa})$ holds and is preserved by $\Add(\kappa^+,1)$ then in the forcing extension $\rfl_{<\rho}(\kappa^+)$ holds as well.
\end{theorem}

\begin{remark}
    Notice that by \cite[Lemma 42]{HAYUTUnger2020} if $\kappa$ is $\kappa^+$-$\Pi^1_1$-subcompact then there is a generic extension in which its $\Pi^1_1$-subcompactness, and in particular $\rfl_{<\kappa}(E^{\kappa^{+}}_{<\kappa})$, is preserved by $\Add(\kappa^+,1)$.
\end{remark}
Let $G\s \mathbb M^\ast$ be a $V$-generic filter and let $x\in\mathbb Y$. We define 
\[g_x:=\{\dot c\in \mathbb X_x\mid \exists A\in U(\kappa) ~(\langle x, A, \dot c\rangle\in G)\}.\]

\begin{lemma}\label{lem: gx generic}
    If $g_x\neq \emptyset$, then $g_x$ is a $\mathbb X_x$-generic filter.
\end{lemma}

Let $D\s \mathbb X_x$ open dense, and let $\dot c\in g_x$. Let $p=\langle x,A,\dot c\rangle\in G$. 
Define 
\[ D':=\{q\in\mathbb M^\ast \mid \exists B\in U(\kappa),\exists d\in D\text{ s.t. } q\le\langle x,B,d\rangle\}.\]
\begin{claim}
    $D'$ is $\mathbb M^\ast$-dense below $p$.
\end{claim}  
\begin{proof}
    Let $q\le p$. Since $q\in\mathbb M^\ast$, $\langle x, A^q\cup \mathbf a(x(q)), c^q\rangle\in \mathbb M^\ast$. 
    Then $c^q\in\mathbb X_x$ and because $D$ is dense in $\mathbb X_x$ there is $d\le c^q$ with $d\in D$. Therefore there is some $B\in U(\kappa)$ with $\langle x, B, d\rangle \le \langle x, B, c^q\rangle$, i.e. it holds that $\langle B\cup \mathbf a(x)\rangle \Vdash \dot c^q \trle \dot d$.
    Notice that $(B\cap A^q)\cup \mathbf a(x(q))\s B\cup \mathbf a(x)$. Therefore if we take $q':=\langle x(q), B\cap A^q, d\rangle$ then, by its definition, $q'\le q,p$ thus $q'\in D'$.
\end{proof}

\begin{proof}[Proof of Lemma \ref{lem: gx generic}]
    Since $D'$ is dense below $p$ then there is some $q\in D'\cap G$.
    Because $q\in D'$ there are some $d\in D$ and $B\in U(\kappa)$ with $q\le \langle x, B,d\rangle$. 
    Let $r:=\langle x, A^q\cup \mathbf a(x(q)),c^q\rangle$, clearly $q\le r\le p$ and thus $r\in G$ so $c^q\in g_x$.
    Since $c^q\le_{\mathbb X_x} d$ it is clear that $c^q\in D$ therefore $c^q\in D\cap g_x$.
\end{proof}

\begin{remark}
    Since $g_x$ is $\mathbb X_x$-generic and $\mathbb D_x$ is a dense $\kappa^+$-closed subset of $\mathbb X_x$, then $\gch$ implies that $\mathbb D_x\simeq \Add(\kappa^+,1)$. 
    Thus in $V[g_x]$ every stationary $S\s E^{\kappa^+}_{<\kappa}$ reflects, because we have assumed that in $V$ every stationary $S\s E^{\kappa^+}_{<\kappa}$ reflects and that this is preserved by $\Add(\kappa^+,1)$.
\end{remark}

Notice that the following lemma completes the proof of Main Theorem \ref{1st maintheorem}.

\begin{lemma} 
    In $V[G]$ every stationary subset of $\kappa^+$ reflects. 
    Moreover if $V$ satisfies $\rfl_{<\kappa} (E^{\kappa^+}_{<\kappa})$, then $V[G]$ satisfies $\rfl_{<\rho}(\kappa^+)$.
\end{lemma}

\begin{proof}
Assume, toward a contradiction, that $S\s (\kappa^+)^V$ is in $V[G]$ stationary and does not reflect.
For all $\alpha\in S$ there is $p_\alpha\in G$ such that $p_\alpha\Vdash \check \alpha\in \dot S$.
Since $|\mathbb Y|=\kappa$ there is $x\in \mathbb Y$ such that
$S':=\{\gamma\mid \exists A\in U(\kappa) \exists c \text{ s.t. } \langle x, A,c\rangle \Vdash \check \gamma\in \dot S\}$ is stationary.

Let $\dot S^\prime$ be an $\mathbb M^\ast$ name for $S'$, we shall try to mirror $S'$ in $V[g_x]$. Define
\[S^\dprime:=\{\alpha \mid \exists q\in \nicefrac{\mathbb M^\ast}{g_x} \text{ s.t. } x(q)=x \text{ and } q\Vdash\check\alpha\in\dot{S^\prime}\}.\]
Notice that $S^\dprime$ reflects in $V[g_x]$, thus let $\beta<\kappa^+$ such that $S^\dprime\cap \beta$ is stationary in $\beta$ and $\cf(\beta)=\eta<\kappa$.
Let $C_\beta\s \beta$ be a club such that $\otp(C_\beta)=\eta$ and assume without loss of generality that $\eta$ is not weakly-compact.
For every $\alpha\in C_\beta\cap S^\dprime$, let $q_\alpha\in\nicefrac{\mathbb M^\ast}{g_x}$ such that $x(q_\alpha)=x$ and $q_\alpha\Vdash\check\alpha\in\dot{S^\prime}$. 
Enumerate $\langle q_\alpha\mid\alpha\in C_\beta\cap S^\dprime\rangle$ and notice that for all $\alpha,\alpha'\in C_\beta\cap S^\dprime$ $q_{\alpha},q_{\alpha'}$ are compatible. 
Since $(\mathbb M\restriction_x,\le^\ast)$ is $<\kappa$-closed and $|C_\beta\cap S^\dprime|=\eta<\kappa$, there is some $q$ extending all $\langle q_\alpha\mid\alpha\in C_\beta\cap S^\dprime\rangle$.
But then $q\Vdash S'\cap C_\beta\neq \emptyset$.\\

Moreover, assume that $\vec S=\langle \dot S_\xi\mid \xi<\theta\rangle$ is such that $1_{\mathbb M^\ast}\Vdash \dot S_\xi\s \check\kappa^+ \text{ is stationary}$ and $\theta<\rho$.
For $\xi<\theta$ let $D(\dot S_\xi):=\{p\in\mathbb M^\ast\mid c^p\Vdash_{\mathbb X_{x(p)}} S^\dprime_{\xi,x(p)} \text{ is stationary} \}$. 
Use the Strong Prikry Property \ref{Strong Prikry Property} for $D(\dot S_0)$ to obtain $p_0\le^\ast 1_{\mathbb M^\ast}$ and $T_0$ a $\kappa$-fat tree.
If for all $\zeta<\xi$ $p_\zeta$ is defined let $p_{\xi}\le^\ast \langle p_\zeta\mid \zeta<\xi\rangle$ and $T_{\xi}$ from applying the Strong Prikry Property \ref{Strong Prikry Property} for $D(\dot S_{\xi})$.
Let $p^\ast\le^\ast\langle p_\xi\mid \xi<\theta\rangle$.

For all $\xi<\theta$ there is $\alpha_\xi$ such that there is $T_\xi^\ast\s T_\xi$ $\kappa$-fat such that for all non-maximal $t\in T_\xi^\ast$ there is $\alpha<\alpha_\xi$ with $\Succ_{T_\xi^\ast}(t)\in U_{\kappa,\alpha}$.
Let $\tilde\alpha=\sup_{\xi<\theta}\alpha_\xi$. 
Therefore there is some $\mu\in A^{p^\ast}$ with $o^{\mathcal {U}}(\mu)=\tilde \alpha$ such that for all $\xi<\theta$ $T_\xi^\ast\cap [\mu]^{<\omega}$ is $\mu$-fat.

Let $p_\mu:=p^\ast +\langle\mu\rangle$ and let $x_\mu:= x(p_\mu)$. 

Work now in $V[g_{x_\mu}]$:
Let $\xi<\theta$. For all possible $t\in T_\xi^\ast$ below $\mu$, define
$S^t_\xi:=\{\alpha<\kappa^+\mid \exists A\in U(\kappa) \exists c\in g_{x_\mu} ~ \langle x_\mu, A, c\rangle + t\Vdash \check \alpha \in \dot S_\xi\}$.
Clearly $S^t_\xi$ is stationary in $V[g_{x_\mu}]$.

Let $\mathcal A:=\{S_\xi^t\mid \xi<\theta \wedge t\in T_\xi^\ast\cap[\mu]^{<\omega}\}$.
Since $|\mathcal A|=\mu<\kappa$ and $V[g_{x_\mu}]$ satisfies $\rfl_{<\kappa} (E^{\kappa^+}_{<\kappa})$ there is some $\delta<\kappa^+$ such that for all $S\in\mathcal A\cup\{E^{\kappa^+}_{\mu^+}\}$ $S\cap\delta$ is stationary and $\cf(\delta)>\mu$.

Let $\eta:=\cf(\delta)$ and let $C_\delta\s \delta$ be a $V$-club with $\otp(C_\delta)=\eta$. 

Let $\xi<\theta$, and let $t\in T_\xi^\ast$ below $\mu$.
Thus for all $\gamma\in S^t_\xi\cap C_\delta$ there is $A(\xi,t,\gamma)\in U(\kappa)$ and there is $c(\xi,t,\gamma)\in g_{x_\mu}$ such that $\langle x_\mu, A(\xi,t,\gamma), c(\xi,t,\gamma)\rangle \Vdash \check\gamma\in \dot S_\xi$.

Let $A^\ast:=\bigcap\{A(\xi,t,\gamma)\mid \xi<\theta, t\in T_\xi^\ast\cap[\mu]^{<\omega}, \gamma\in S^t_\xi\cap C_\delta\}$, and let $c^\ast\le\langle c(\xi,t,\gamma)\mid \xi<\theta, t\in T_\xi^\ast\cap[\mu]^{<\omega}, \gamma\in S^t_\xi\cap C_\delta\rangle$.

Thus $p^{\ast}_\mu:=\langle x_\mu, A^\ast, c^\ast\rangle\le^\ast p_\mu$.
Hence, for every maximal $t\in T_\xi^\ast$ below $\mu$ that is possible for $p^\ast_\mu$, $p^\ast_\mu+t\Vdash \dot S_\xi\cap C_\delta \text{ is stationary}$. Therefore for every $\xi<\theta$ $p^\ast_\mu \Vdash \dot S_\xi\cap \check \delta \text{ is stationary}$.
\end{proof}

\subsection*{\texorpdfstring{Achieving stationary reflection at $\aleph_{\nu+1}$}{Achieving stationary reflection at aleph(nu+1)}}\label{subsec: downto getting stat}

\begin{theorem}[Main Theorem \ref{2nd mainthorem}] \label{main theorem 2}
Assume $\gch$.
    Let $\rho$ be some regular cardinal.
    If in $V$ $\kappa$ is measurable with $o(\kappa)\geq \rho$ and $\rfl(E^{\kappa^+}_{<\kappa})$ holds and is preserved by $\Add(\kappa^+,1)$ and 
    for any stationary $S\s E^{\kappa^+}_{<\kappa}$, 
    there are unboundedly many regular cardinals $\mu < \kappa$ such that $T(\{S\}) \cap E^{\kappa^{+}}_\mu \neq \emptyset$ and one of the following holds:
    \begin{enumerate}        
        \item $\mu$ is not a successor of a singular cardinal, 
        \item $\ap_{\mu}$ holds. 
    \end{enumerate}
    Then there is a forcing extension in which $\kappa=\aleph_{\rho}$, $\kappa^+=\aleph_{\rho+1}$ and in which $\rfl(\aleph_{\rho+1})$ holds. 
\end{theorem}

\begin{proof}
     Let $K:=\{\alpha<\kappa\mid \exists p\in G~(\alpha\in\dom(p))\}$ and let $K_0:=\{\alpha\in K\mid o^{\mathcal {U}}(\alpha)=0\}$. 
        
    Assume on the contrary that $S\s (\kappa^+)^V$ is in $V[G]$ stationary and does not reflect. For all $\alpha\in S$ there is $p_\alpha\in G$ such that $p_\alpha\Vdash \check \alpha\in \dot S$. 
    Since $|\mathbb Y'|=\kappa$ there is $y\in \mathbb Y'$ such that
    $S':=\{\gamma\mid \exists I\in \mathcal G(\kappa) \exists C \text{ s.t. } \langle y, I,C\rangle \Vdash \check \gamma\in \dot S\}$ is stationary.
    Let $x:=\pi^1[y]$.

    Let $\dot S'$ be an $\mathbb R\ast$ name for $S'$, we shall try to mirror $S'$ in $V[g_x]$. Since $y\in V[g_x]$, define 
    \[S^\dprime_y:=\{\alpha \mid \exists q\in \nicefrac{\mathbb R^\ast}{g_x} \text{ s.t. } y(q)=y \text{ and } q\Vdash\check\alpha\in\dot{S^\prime}\}.\]
    Since $S^\dprime_y$ is stationary in $V[g_x]$, there is some $\beta<\kappa$ such that $S^\dprime_y\cap S_\beta^{\kappa^+}$ is stationary. Take $S_y:=S^\dprime_y\cap S_\beta^{\kappa^+}$.

    Let $\gamma_0=\min \{\xi\in K\mid \xi>\beta\}$ and let $\gamma_1=\min\{\xi\in K_0\mid \xi> \gamma_0 \}$. 
    Notice that $S_y$ reflects in $V[g_x]$, thus pick 
    $\eta \in T(\{S_y\})\setminus \gamma_1$.
        
    \begin{claim}
        If $\ap_{\alpha^+}$ holds for all singular cardinal $\alpha$ then $S_y\cap \delta$ remains stationary in $V[G]$.
    \end{claim}

    \begin{proof}
        Suppose that $\eta=\alpha^+$ for some singular cardinal $\alpha$.
        Let $C(\delta)\s \delta$ be a club in $V$ such that $\otp(C(\delta))=\eta$.
        Let $\hat S=\{ i<\eta\mid C(\delta)(i)\in S_y\}$. Clearly $\hat S\s E^\eta_\beta$ stationary.
        Assume $p\Vdash\dot x\s \check\eta \text{ is a club}$.
        Let $p\in G$ such that $\gamma_0,\gamma_1\in \dom (p)$, $\beta<\gamma_0<\eta<\gamma_1$, $\gamma_1=\min (\dom(p)\setminus \gamma_0+1)$ and $o^{\mathcal {U}}(\gamma_1)=0$.
        By \ref{bounded new subsets below kappa are from R} there is $\tilde y$ 
        an $\mathbb R\restriction_{\gamma_1}$ name and $p'\in G$ such that 
        $p'\Vdash \dot x=\tilde y$.
        Since $\ap_\eta$ holds $\hat S\in I[\eta]$, i.e. 
        there are a sequence $\langle b_\xi\mid \xi<\eta\rangle$ and a club $X\s \eta$ such that for all $\xi\in X\cap \hat S$ there is some $x_\xi\s\xi$ unbounded, $\otp(x_\xi)=\cf(\xi)$ with the property
        $\forall\zeta\in x_\xi$, $x_\xi\cap \zeta\in\{b_\varepsilon\mid \varepsilon<\xi\}$.

        By \cite{Sh:108} $p\restriction_{\gamma_1} \Vdash \tilde y\cap \check{\hat S}\neq \emptyset$.
        Therefore $p'\Vdash \dot x\cap \check{\hat S}\neq\emptyset$.

        Thus $S_y\cap\delta$ remains stationary in $V[G]$.
    \end{proof}

    If $\ap_{\xi^+}$ does not hold for all singular cardinals $\xi<\kappa$, then there is some $\delta<\kappa^+$ with $\cf (\delta)=\eta$ such that $\eta>\beta$ is a not a successor of a singular cardinal and in $V[g_x]$ $S_y\cap \delta$ is stationary in $\delta$.

    \begin{claim}
        $S_y\cap \delta$ remains stationary in $V[G]$.
    \end{claim}
    \begin{proof}
        Let $C(\delta)\s \delta$ a be club in $V$ such that $\otp(C(\delta))=\eta$.
        Let $\hat S=\{ i<\eta\mid C(\delta)(i)\in S_y\}$. Clearly $\hat S\s E^\eta_\beta$ stationary.
        Let $\eta=\cf^V(\delta)$, and assume $p\Vdash\dot x\s \check\eta \text{ is a club}$.
        Let $p\in G$ such that $\gamma_0,\gamma_1\in \dom (p)$, $\gamma_0<\eta<\gamma_1$, $\gamma_1=\min (\dom(p)\setminus \gamma_0+1)$ and $o^{\mathcal {U}}(\gamma_1)=0$.
        By \ref{bounded new subsets below kappa are from R} there is 
        an $\mathbb R\restriction_{\gamma_1}$ name $\tilde y$ and $p'\in G$ such that 
        $p'\Vdash \dot x=\tilde y$.
        By \cite{Sh:108} $S_\beta^\eta\in I[\eta]$ thus $p\restriction_{\gamma_1} \Vdash \tilde y\cap \check{\hat S}\neq \emptyset$.
        Therefore $p'\Vdash \dot x\cap \check{\hat S}\neq\emptyset$.

        Thus $S_y\cap\delta$ remains stationary in $V[G]$.
    \end{proof}
\end{proof}

\begin{theorem}[Main Theorem \ref{3rd mainthorem}] \label{main theorem 3}
   Assume $\gch$.
    Let $\rho$ be some regular cardinal.
    If in $V$ $\kappa$ is measurable with $o(\kappa)\geq \rho$ and $\rfl_{<\kappa}(E^{\kappa^+}_{<\kappa})$ holds and is preserved by $\Add(\kappa^+,1)$ and 
    for any collection $\mathcal S\s\mathcal P(E^{\kappa^+}_{<\kappa})$ consisting $<\kappa$ stationary subsets of $E^{\kappa^+}_{<\kappa}$, 
    there are unboundedly many regular cardinals $\mu < \kappa$ such that $T(\mathcal S) \cap E^{\kappa^{+}}_\mu \neq \emptyset$ one of the following holds:
    \begin{enumerate}        
        \item $\mu$ is not a successor of a singular cardinal, 
        \item $\ap_{\mu}$ holds. 
    \end{enumerate}
    Then there is a forcing extension in which 
    $\kappa=\aleph_{\rho}$, $\kappa^+=\aleph_{\rho+1}$ and in which $\rfl_{<\rho}(\aleph_{\rho+1})$ holds.
    \end{theorem}

    \begin{proof}
    Assume that in $V$ $\rfl_{<\kappa}(E^{\kappa^+}_{<\kappa})$ holds.
    
    Let $\theta<\rho$, and let $\vec S=\langle \dot S_\xi\mid \xi<\theta\rangle$ be such that 
        $1_{\mathbb R^\ast}\Vdash \dot S_\xi\s \check\kappa^+ \text{ is stationary}$.
        For $\xi<\theta$ let 
        \[D(\dot S_\xi):=\{p\in\mathbb R^\ast\mid \dot c^p\Vdash_{\mathbb X_{x(p)}} S^\dprime_{\xi,y(p)} \text{ is stationary} \}.\]
        Use the Strong Prikry Property \ref{Strong R* Prikry Property} for $D(\dot S_0)$ to obtain $p_0\le^\ast 1_{\mathbb R^\ast}$ and $T_0$ a $\kappa$-fat tree. 
        If $p_\xi$ is defined, obtain $p_{\xi+1}\le^\ast p_\xi$ and  $T_{\xi+1}$ by applying the Strong Prikry Property \ref{Strong R* Prikry Property} for $D(\dot S_{\xi+1})$.
        If $p_\zeta$ has been defined for every $\zeta<\xi$, 
        let $p^\ast\le^\ast\langle p_\xi\mid \xi<\theta\rangle$.
        Let us obtain $p_\xi\le^\ast \langle p_\zeta\mid \zeta<\xi\rangle$ and $T_{\xi}$ by applying the Strong Prikry Property \ref{Strong R* Prikry Property} for $D(\dot S_{\xi})$. 
        
        Notice that for all $\xi<\theta$ there is $\alpha_\xi$ such that there is $T_\xi^\ast\s T_\xi$ $\kappa$-fat such that for all $t\in T_\xi^\ast$ non-maximal there is $\alpha<\alpha_\xi$ with $\Succ_{T_\xi^\ast}(t)\in U_{\kappa,\alpha}$.
        Let $\tilde\alpha=\sup_{\xi<\theta}\alpha_\xi$. 
        Therefore there is some $\mu\in A^{p^\ast}$ with $o^{\mathcal {U}}(\mu)=\tilde \alpha$ such that for all $\xi<\theta$ $T_\xi^\ast\cap [\mu]^{<\omega}$ is $\mu$-fat.
        Let $p_\mu:=p^\ast +\langle\mu\rangle$ and let $x_\mu:= \pi^1(y(p_\mu))$.

        Work now in $V[g_{x_\mu}]$:
        Let $\xi<\theta$. For all possible $t\in T_\xi^\ast$ below $\mu$, define
        $S^t_\xi:=\{\alpha<\kappa^+\mid \exists A\in U(\kappa)\exists i<\kappa^+ \exists \dot c\in g_{x_\mu} ~ \langle y(p_\mu), (A,i), \dot c\rangle + \langle t\rangle\Vdash \check \alpha \in \dot S_\xi\}$.
        Clearly $S^t_\xi$ is stationary in $V[g_{x_\mu}]$.
        
        Let 
        $\mathcal A:=\{S_\xi^t\mid \xi<\theta \wedge t\in T_\xi^\ast\cap[\mu]^{<\omega}\}$.
        Since $S_\xi^t$ is stationary in $V[g_{x_\mu}]$, there is some $\beta(\xi,t)<\kappa$ such that $S_\xi^t\cap E_{\beta(\xi,t)}^{\kappa^+}$ is stationary. Replace $S_\xi^t$ by this stationary subset.
        Let 
        \[\gamma_0=\min \{\zeta\in K\mid \zeta>\sup(\{\beta(\xi,t)\mid \xi<\theta, t\in T^\ast_\xi\cap [\mu]^{<\omega}\})\}\] 
        and 
        \[\gamma_1=\min\{\zeta\in K_0\mid \zeta> \gamma_0 \}.\]
        Since $|\mathcal A|=\mu<\kappa$ and since $\rfl_{<\kappa}( E^{\kappa^+}_{<\kappa})$ holds and preserved by $\Add (\kappa^+,1)$ it also holds in $V[g{x_\mu}]$.

        \begin{claim}
            There are $\delta<\kappa^+$ of cofinality $\eta>\gamma_1$ and $\{\zeta_\varepsilon\mid \varepsilon<\eta\}$ a club in delta such that every $S\in\mathcal A$ reflects at $\delta$ and $\{ \varepsilon<\eta\mid \zeta_\varepsilon \in S\}\in I[\eta]$.
        \end{claim}
        \begin{proof}
            If $\ap_{\alpha^+}$ holds for all $\alpha$ singular then there is some $\delta<\kappa^+$ with $\cf(\delta)=\eta>\gamma_1$ such that every $S\in\mathcal A$ reflects at $\delta$, 
            and if $\{\zeta_\varepsilon\mid \varepsilon<\eta\}$ is a club in $\delta$ then clearly  $\{ \varepsilon<\eta\mid \zeta_\varepsilon \in S\}\in I[\eta]$.

            Otherwise the set $\{\eta<\kappa\mid \eta \text{ inaccessible, } \exists\delta \in E^{\kappa^+}_{\eta^+} \text{ s.t. } \forall S\in \mathcal A ~ S\cap\delta \text{ is stationary}\}$ is unbounded in $\kappa$. Thus there is some $\delta<\kappa^+$ with $\cf(\delta)=\eta$ for some regular non-successor of singular $\eta>\gamma_1$. Clearly if $\{\zeta_\varepsilon\mid \varepsilon<\eta^+\}$ is a club in $\delta$ then for all $S\in\mathcal A$ the set $\{\varepsilon<\eta^+\mid \zeta_\varepsilon\in S \}\in I[\eta^+]$.
        \end{proof}

        Therefore let $\delta<\kappa^+$ with $\cf(\delta)=\eta>\gamma_1$ and let $C_\delta=\{\zeta_\varepsilon\mid \varepsilon<\eta\}$ be some $V$-club in $\delta$ such that for all $S\in\mathcal A$  $\{\varepsilon<\eta\mid \zeta_\varepsilon\in S\}\in I[\eta]$.

        Let $\xi<\theta$ and let $t\in T_\xi^\ast$ below $\mu$. Then for all $\gamma\in S^t_\xi\cap C_\delta$ there are $i(\xi,t,\gamma)<\kappa^+$, $A(\xi,t,\gamma)\in \bigcap_{j<o^{\mathcal{ U}}(\kappa)} U(\kappa,j)$ and $\dot c(\xi,t,\gamma)\in g_{x_\mu}$ such that 
        \[\langle y_\mu, (A(\xi,t,\gamma),i(\xi,t,\gamma)), c(\xi,t,\gamma)\rangle \Vdash \check\gamma\in \dot S_\xi.\]
        Let 
        \[A^\ast:=\bigcap\{A(\xi,t,\gamma)\mid \xi<\theta, t\in T_\xi^\ast\cap[\mu]^{<\omega}, \gamma\in S^t_\xi\cap C_\delta\},\] 
        \[i^\ast:=\sup\{i(\xi,t,\gamma)\mid \xi<\theta, t\in T_\xi^\ast\cap[\mu]^{<\omega}, \gamma\in S^t_\xi\cap C_\delta\},\] 
        and  
        \[\dot c^\ast\le\langle \dot c(\xi,t,\gamma)\mid \xi<\theta, t\in T_\xi^\ast\cap[\mu]^{<\omega}, \gamma\in S^t_\xi\cap C_\delta\rangle.\]
        Now take $p^{\ast}_\mu:=\langle y_\mu, (A^\ast, i^\ast), \dot c^\ast\rangle\le^\ast p_\mu$.

        Hence for every maximal $t\in T_\xi^\ast$ below $\mu$ that is possible for $p^\ast_\mu$, $p^\ast_\mu+t\Vdash S_\xi^t\cap C_\delta \s \dot S_\xi\cap C_\delta$.

        Let $G\s \mathbb R^\ast$ such that $p^\ast_\mu\in G$.
        
    \begin{claim}
            For all $\xi<\theta$ and $t$ possible for $p^\ast_\mu$, $S_\xi^t\cap C_\delta$ remains stationary in $V[G]$.
    \end{claim}
        
    \begin{proof}
        Let $\hat S^t_\xi:=\{\varepsilon<\eta\mid \zeta_\varepsilon\in S^t_\xi\}$, and from the choice of $\delta,\eta,C_\delta$ we have that $\hat S^t_\xi\in I[\eta]$.
        Thus for all $q\le p^\ast_\mu+t$ there is some $q'\le q$ and some $\alpha_0,\alpha_1\in\dom (q')$ such that $\alpha_0<\eta<\alpha_1$ and $o^{\mathcal {U}}(\alpha_1)=0$.
        By Lemma~\ref{bounded new subsets below kappa are from R} and \cite{Sh:108}, $q'$ forces that $\hat S_\xi^t$ remains stationary.
        Thus $p^\ast_\mu\Vdash \hat{S}^t_\xi \text{ remains stationary}$.
    \end{proof}
    Therefore for every $\xi<\theta$ $p^\ast_\mu \Vdash \dot S_\xi\cap \check \delta \text{ is stationary}$.
\end{proof}

\section*{Remarks and further developments}
Our next step is to determine whether the same methods work together with the failure of $\sch$--- this is currently a work in progress.
\section*{Acknowledgments}
The author is deeply grateful to their advisor, Yair Hayut, for their invaluable guidance, insight, and continuous support throughout this research. The author is supported by the Israel Science Foundation, Grant No 1967/21.

\end{document}